\documentclass[leqno]{article} 

\usepackage[frenchb,english]{babel}
\usepackage[T1]{fontenc}

\usepackage{amsmath}
\usepackage{amssymb}
\usepackage{amsfonts}
\usepackage{enumerate}
\usepackage{vmargin}
\usepackage[all]{xy}
\usepackage{mathrsfs}
\usepackage{mathtools}
\usepackage{lmodern}
\usepackage{hyperref}

\setmarginsrb{3cm}{3cm}{3.5cm}{3cm}{0cm}{0cm}{1.5cm}{3cm}

\newcommand{\N}{\ensuremath{\mathbb{N}}}

\newcommand{\Mat}{\ensuremath{\mathbb{M}}}

\newcommand{\R}{\ensuremath{\mathbb{R}}}
\newcommand{\C}{\ensuremath{\mathbb{C}}}

\newcommand{\VN}{{\rm VN}}

\renewcommand{\leq}{\ensuremath{\leqslant}}
\renewcommand{\geq}{\ensuremath{\geqslant}}
\newcommand{\n}{\noindent}
\newcommand{\qed}{\hfill \vrule height6pt  width6pt depth0pt}
\newcommand{\norm}[1]{ \| #1  \|}
\newcommand{\bnorm}[1]{ \big\| #1  \big\|}
\newcommand{\Bnorm}[1]{ \Big\| #1  \Big\|}
\newcommand{\bgnorm}[1]{ \bigg\| #1  \bigg\|}
\newcommand{\Bgnorm}[1]{ \Bigg\| #1  \Bigg\|}
\newcommand{\xra}{\xrightarrow}

\newcommand{\ot}{\otimes}
\newcommand{\ovl}{\overline}

\newcommand{\ul}{\mathcal{U}}

\newcommand{\reg}{{\rm reg}}
\newcommand{\co}{\colon}

\newtheorem{thm}{Theorem}[section]

\newtheorem{prop}[thm]{Proposition}

\newtheorem{lemma}[thm]{Lemma}

\newtheorem{remark}[thm]{Remark}

\newenvironment{proof}[1][]{\noindent {\it Proof #1} : }{\hbox{~}\qed
\smallskip
}

\numberwithin{equation}{section}

\begin{document}
\selectlanguage{english}
\title{\bfseries{On a conjecture of Pisier on the analyticity of semigroups}}
\date{}
\author{\bfseries{Cédric Arhancet}}

\maketitle

\begin{abstract}
We show that the analyticity of semigroups $(T_t)_{t \geq 0}$ of selfadjoint contractive Fourier multipliers on $L^p$-spaces of compact abelian groups is preserved by the tensorisation of the identity operator of a Banach space for a large class of K-convex Banach spaces, answering partially a conjecture of Pisier. We also give versions of this result for some semigroups of Schur multipliers and Fourier multipliers on  noncommutative $L^p$-spaces. Finally, we give a precise description of semigroups of Schur multipliers to which the result of this paper can be applied.
\end{abstract}


\makeatletter
 \renewcommand{\@makefntext}[1]{#1}
 \makeatother
 \footnotetext{
The author is supported by the research program ANR 2011 BS01 008 01. 
}

\section{Introduction}

In the early eighties, in a famous paper on the geometry of Banach spaces, Pisier \cite[Theorem~2.1]{Pis3} showed that a Banach space $X$ does not contain $\ell^1_n$'s uniformly if and only if the tensorisation $P \ot Id_X$ of the Rademacher projection
\begin{equation*}
\begin{array}{cccc}
   P   \co &  L^2(\Omega)   &  \longrightarrow   &  L^2(\Omega)  \\
           &   f  &  \longmapsto       & \displaystyle \sum_{k=1}^\infty \Big(\int_\Omega f\varepsilon_k \Big)\varepsilon_k \\
\end{array}
\end{equation*}
induces a bounded operator on the Bochner space $L^2(\Omega,X)$ where $\Omega$ is a probability space and where $\varepsilon_1,\varepsilon_2,\ldots$ is a sequence of independant random variables with $P(\varepsilon_k=1)=P(\varepsilon_k=-1)=\frac{1}{2}$. Such a Banach space $X$ is called K-convex. The heart of his proof relies on the fact, proved by himself in his article, that if $X$ is a K-convex Banach space then any $w^*$-continuous semigroup $(T_t)_{t \geq 0}$ of positive unital selfadjoint Fourier multipliers on a locally compact abelian group $G$ induces a strongly continuous bounded analytic semigroup $(T_t \ot Id_X)_{t \geq 0}$ of contractions on the Bochner space $L^p(G,X)$ for any $1 < p < \infty$. In 1981, in the seminars {\cite{Pis1} and \cite{Pis2} which announced the results of his paper, he stated several natural questions raised by his work. In particular, he conjectured \cite[page 17]{Pis1} that the same property holds for any $w^*$-continuous semigroup $(T_t)_{t \geq 0}$ of selfadjoint contractive operators on $L^\infty(\Omega)$ where $\Omega$ is a measure space (see also the recent preprint \cite[Problem 11]{Xu} for a more general question). Note that it is well-known \cite[III2 Theorem 1]{Ste} that such a semigroup induces a strongly continuous bounded analytic semigroup of contractions on the associated $L^p$-space $L^p(\Omega)$ and the conjecture says that the property of analyticity is preserved by the tensorisation of the identity $Id_X$ of a K-convex Banach space $X$. 

Using operator space theory (see \cite{ER}, \cite{Pau} and \cite{Pis7}), a quantised theory of Banach spaces, we are able to give the following partial answer to this purely Banach spaces question. First of all, let us recall that an operator space $E$ is OK-convex if the vector valued Schatten space $S^p(E)$ is K-convex for some (equivalently all) $1<p<\infty$. It means that the Rademacher projection $P$ is completely bounded. This notion was introduced by \cite{JP} and is the noncommutative version of the property of K-convexity. Our main result is the following theorem.


\begin{thm}
\label{Th main commutative}
Suppose that $G$ is a compact abelian group. Let $(T_t)_{t \geq 0}$ be a $w^*$-continuous semigroup of selfadjoint contractive Fourier multipliers on $L^\infty(G)$. Let $X$ be a K-convex Banach space isomorphic to a Banach space $E$ which admits an OK-convex operator space structure. Consider $1<p<\infty$. Then $(T_t)_{t\geq 0}$ induces a strongly continuous  bounded analytic semigroup $(T_t \ot Id_X)_{t \geq 0}$ of contractions on the Bochner space $L^p(G,X)$.
\end{thm}
This result can be used, by example, in the case
where the Banach space $X$ is an $L^q$-space or a Schatten space $S^q$ with $1<q<\infty$. Our methods also give a result for some $w^*$-continuous semigroups of Schur multipliers and a generalization for semigroups of Fourier multipliers on amenable discrete groups. See also the forthcoming paper \cite{Arh3} for related results.

The paper is organized as follows. Section 2 gives a brief presentation of vector valued noncommutative $L^p$-spaces, Fourier multipliers on group von Neumann algebras and Schur multipliers. We introduce here some notions which are relevant to our paper. The next section 3 contains a proof of Theorem \ref{Th main commutative}. Finally, in Section 4, we describe the semigroups of Schur multipliers to which the results of this paper can be applied. This result is of independent interest.

\section{Preliminaries}
The readers are referred to \cite{ER}, \cite{Pau} and \cite{Pis7}
for details on operator spaces and completely bounded maps and to
the survey \cite{PiX} for noncommutative $L^p$-spaces and the references therein.

If $T\co E \to F$ is a completely bounded map between two operators spaces $E$ and $F$, we denote by $\norm{T}_{cb, E \to F}$ its completely bounded norm.

The theory of vector valued noncommutative $L^p$-spaces was initiated by Pisier \cite{Pis5} for the case where the underlying von Neumann algebra is hyperfinite and equipped with a normal semifinite faithful trace. Suppose $1 \leq p < \infty$. Under theses assumptions, for any operator space $E$, we can define by complex interpolation
\begin{equation}
\label{Def vector valued Lp non com}
L^p(M,E)=\big(M \ot_{\min} E, L^1(M)\widehat{\ot}
E\big)_{\frac{1}{p}},
\end{equation}
where $\ot_{\min}$ and $\widehat{\ot}$ denote the injective and the projective tensor product of operator spaces. 

If $I$ is an index set then we denote by $B(\ell^2_I)$ the von Neumann algebra of bounded operators on the Hilbert space $\ell^2_I$. Using its canonical trace, we obtain the vector valued Schatten space $S^p_I(E)=L^p(B(\ell^2_I),E)$. With $E=\C$, we recover the classical Schatten space $S^p_I$. Sometimes, we will use the notation $S^p(B(\ell^2_I)^{\ot n},E)$ for the space $L^p(B(\ell^2_I)^{\ot n},E)$.

Note the following extension properties of some linear maps between noncommutative $L^p$-spaces, see \cite{Pis4}, \cite[Lemma 6.1]{Pis8} and \cite[Lemma 4.1]{Arh1}.
\begin{prop}
\label{prop-tensorisation of CP maps}
Let $M$ and $N$ be von Neumann algebras equipped with normal semifinite faithful traces.
\begin{enumerate}
	\item Let $T \co M \to N$ be a trace preserving unital normal completely positive map. Suppose $1 \leq p < \infty$. Then $T$ induces a complete contraction $T \co L^p(M) \to L^p(N)$.
	\item Suppose that $M$ and $N$ are hyperfinite. Let $E$ be an operator space. Let $T \co M \to N$ be a complete contraction that also induces a complete contraction on $L^1(M)$. Suppose $1 \leq p \leq \infty$. Then the operator $T \ot Id_E$ induces a completely contractive operator from $L^p(M,E)$ into $L^p(N,E)$.
\end{enumerate}
\end{prop}

In particular, this result applies to canonical normal conditional expectations between von Neumann algebras (see \cite[Theorem 10.1]{Str}). If $1<p<\infty$, recall that a linear map $T \co L^p(M) \to L^p(N)$ between noncommutative $L^p$-spaces on hyperfinite von Neumann algebras (equipped with normal semifinite faithful traces) is said to be regular \cite{Pis4} if for any operator space $E$ the linear map $T \ot Id_E$ induces a bounded operator from $L^p(M,E)$ into $L^p(N,E)$. The norm $\norm{T}_{\reg}$ denote the best constant $C$ such that $\norm{T \ot Id_E}_{L^p(M,E) \to L^p(N,E)} \leq C$ for any operator space $E$. The above proposition gives examples of contractively regular maps (i.e. $\norm{T}_{\reg} \leq 1$). 

Suppose $1 \leq p<\infty$. Let $(T_t)_{t \geq 0}$ be a $C_0$-semigroup of contractively regular operators on a noncommutative $L^p$-space $L^p(M)$ of a hyperfinite von Neumann algebra $M$. Then, using \cite[Proposition 5.3]{EN}, it is not difficult to prove that for any operator space $E$ the semigroup $(T_t \ot Id_{E})_{t \geq 0}$ of contractive operators acting on the vector valued $L^p$-space $L^p(M,E)$ is strongly continuous.

%

%

Suppose that $G$ is a discrete group. We denote by $e_G$ the neutral element of $G$. For $f \in \ell^1_G$, we write $L_f$ for the left convolution by $f$ acting on $\ell^2_G$ by:
$$
L_f (h)(g') =\sum_{g\in G} f(g)h(g^{-1}g')
$$
where $h \in \ell^2_G$ and $g' \in G$. Let $\VN(G)$ be the von Neumann algebra generated by the set $\{L_f\ :\ f \in \ell^1_G\}$. It is called the group von Neumann algebra of $G$ and is equal to the von Neumann algebra generated by the set $\{\lambda_g \ :\ g \in G\}$, where $\lambda_g$ is the left translation acting on $\ell^2_G$ defined by $\lambda_g(h)(g')=h(g^{-1}g')$. It is an finite algebra and its normalized normal finite faithful trace is given by
$$
\tau_{G}(x)=\big\langle\epsilon_{e_G},x(\epsilon_{e_G})\big\rangle_{\ell^2_G}
$$
where $(\epsilon_g)_{g\in G}$ is the canonical basis of $\ell^2_G$ and $x \in \VN(G)$. Recall that the von Neumann algebra $\VN(G)$ is hyperfinite if and only if $G$ is amenable \cite[Theorem 3.8.2]{SS}.

For a locally compact abelian group $G$, the Fourier transform of $f \in L^1(G)$ is defined on the dual group $\hat{G}$ of $G$ by
$$
\hat{f}(\gamma) =\int_G f(x)\ovl{\gamma(x)}dx
$$
for $\gamma \in \hat{G}$. Moreover, it is well known that a locally compact abelian group $G$ is discrete if and only if the dual group $\hat{G}$ is compact.

If $G$ is a discrete abelian group then ($\VN(G)$, $\tau_G)$ is equivalent as a von Neumann algebra to $L^\infty(\hat{G})$ with the usual integration on the dual group $\hat{G}$ of $G$ under the mapping 
\begin{equation}
\label{Case abelian AVN}
\begin{array}{cccc}
           &  \VN(G)   &  \longrightarrow   &  L^\infty(\hat{G})  \\
           &  L_f   &  \longmapsto       &  \hat{f}.  \\
\end{array}
\end{equation}

Let $G$ be a discrete group. A Fourier multiplier on $\VN(G)$ is a normal linear map $T \co \VN(G) \to \VN(G)$ such that there exists a complex function $\varphi \co G \to \C$ such that for any $g \in G$ we have $T(\lambda_g)=\varphi_g \lambda_g$. In this case, we also denote $T$ by
$$
\begin{array}{cccc}
   M_\varphi :   &    \VN(G)      &  \longrightarrow   & \VN(G)   \\
          &    \lambda_g  &  \longmapsto       & \varphi_g\lambda_g.   \\
\end{array}
$$
If the discrete group $G$ is amenable then every contractive Fourier multiplier $ M_\varphi \co \VN(G) \to \VN(G)$ is completely contractive, see \cite[Corollary 1.8]{DCH}, \cite[Theorem 1]{Los} and \cite[Corollary 4.3]{Spr}.

If $I$ is an index set and if $E$ is a vector space, we write
$\Mat_I$ for the space of the $I \times I$ matrices with entries in
$\C$ and $\Mat_I(E)$ for the space of the $I \times I$ matrices with
entries in $E$. 

Let $A=[a_{ij}]_{i,j\in I}$ be a matrix of $\mathbb{M}_I$. By definition, the Schur multiplier on $B\big(\ell^2_I\big)$ associated with this matrix is the unbounded linear operator $M_A$ whose domain $D(M_A)$ is the space of all $B=[b_{ij}]_{i,j\in I}$ of $B\big(\ell^2_I\big)$ such that $[a_{ij}b_{ij}]_{i,j\in I}$ belongs to $B\big(\ell^2_I\big)$, and whose action on $B=[b_{ij}]_{i,j\in I}$ is given by $M_A(B)=[a_{ij}b_{ij}]_{i,j\in I}$. For any $i,j \in I$, the matrix $e_{ij}$ belongs to $D(M_A)$, hence $M_A$ is densely defined for the weak* topology. Suppose $1\leq p \leq \infty$. If for any $B \in S^p_{I}$, we have $B \in D(M_A)$ and the matrix $M_{A}(B)$ represents an element of $S^p_{I}$, by the closed graph theorem, the matrix $A$ of $\Mat_{I}$ defines a bounded Schur multiplier $M_A \co S^p_{I} \to S^p_{I}$. We have a similar statement for bounded Schur multipliers on $B\big(\ell^2_I\big)$. Recall that every contractive Schur multiplier $M_A \co B\big(\ell^2_I\big) \to B\big(\ell^2_I\big)$ is completely contractive \cite[Corollary 8.8]{Pau}).
It is well-known that a matrix $A$ of $\mathbb{M}_I$ induces a completely positive (or positive) Schur multiplier $M_A \co B\big(\ell^2_I\big) \to B\big(\ell^2_I\big)$ if and only if for any finite set $ F\subset I$ the matrix $[a_{i,j}]_{i,j \in F}$ is positive, see \cite[Exercice 8.7]{Pau}) and \cite[Theorem C.1.4]{BHV}. 

Let $M$ be a von Neumann algebra equipped with a normal semifinite faithful trace $\tau$. Suppose that $T \co M \to M$ is a normal contraction. We say that $T$ is selfadjoint if for any $x,y \in M \cap L^1(M)$ we have
$$
\tau\big(T(x)y^*\big)=\tau\big(x(T(y))^*\big).
$$
In this case, it is not hard to show that the restriction $T|M \cap
L^1(M)$ extends to a contraction $T \co L^1(M) \to L^1(M)$. By complex interpolation, for any $ 1\leq p <\infty$, we obtain a contractive map $T \co L^p(M) \to L^p(M)$. Moreover, the operator $T \co L^2(M) \to L^2(M)$ is selfadjoint. If $T \co M \to M$ is actually a normal selfadjoint complete contraction, it is easy to see that the map $T \co L^p(M) \to L^p(M)$ is completely contractive for any $1 \leq p < \infty$. It is not difficult to show that a contractive Fourier multiplier $M_\varphi \co \VN(G) \to \VN(G)$ is selfadjoint if and only if $\varphi \co G \to \C$ is a real function. Finally, one can prove that a contractive Schur multiplier $M_A$ is selfadjoint if and only if $A$ is a real matrix.

Let $X$ be a non-empty set. A kernel on $X$ is a function $\psi \co X \times X \to \C$. The kernel is called hermitian if $\psi(y,x) = \ovl{\psi(x,y)}$ for any $x,y \in X$. It is positive definite \cite[Definition 1.1, page 67]{BCR} \cite[Definition C.1.1]{BHV} if for any integer $n \in \N$, any $x_1,\ldots,x_n\in X$ and any $c_1,\ldots,c_n\in\C$ we have
$$
\sum_{i,j=0}^n c_i \ovl{c_j}\psi(x_i,x_j) \geq 0.
$$
It is called conditionally negative definite \cite[Definition C.2.1]{BHV} \cite[Definition 1.1, page 67]{BCR} if it is hermitian and if for any integer $n \in \N$, any $x_1,\ldots,x_n \in X$ and any $c_1,\ldots,c_n \in \C$ such that $\sum_{i=1}^n c_i = 0$ we have
$$
\sum_{i,j=1}^n c_i \ovl{c_j}\psi(x_i,x_j) \leq 0.
$$
The set of all conditionally of negative type kernels on $X$ is a convex
cone, that is, if $\psi_1$ and $\psi_2$ are kernels conditionally of negative type then so is $s\psi_1+t\psi_2$ for all positive real numbers $s, t \geq 0$. If $\psi$ is a positive definite kernel, it is obvious that $-\psi$ is a conditionally negative definite kernel. Moreover, any constant kernel $\psi \co X \times X \to \C$ is conditionally of negative type. By \cite[Proposition 3.2, page 82]{BCR} (or \cite[Theorem C.2.3]{BHV}), if $\psi \co X \times X \to \R$ is a real-valued conditionally negative definite kernel that vanishes on the diagonal $\{(x,x) \mid x\in X\}$ then there exist a real Hilbert space $H$ and a map $\xi \co X \to H$ such that
$$
k(x,y) =\bnorm{\xi(x)-\xi(y)}_{H}^2.
$$
\section{Analyticity of semigroups on vector valued $L^p$-spaces}

Let $X$ be a Banach space. A strongly continuous semigroup
$(T_t)_{t\geq 0}$ is called bounded analytic if there exist
$0<\theta<\frac{\pi}{2}$ and a bounded holomorphic extension
$$
\begin{array}{cccc}
    &  \Sigma_\theta   &  \longrightarrow   &  B(X)  \\
    &  z   &  \longmapsto       &  T_z  \\
\end{array}
$$
where $\Sigma_\theta=\{z\in\C^*\ :\ \vert{\rm Arg}(z)\vert <\theta\}$ denotes the open sector of angle $2\theta$ around the positive real axis $\R_+$. See \cite{ABHN}, \cite{EN} and \cite{Haa} for more information on this notion. We need the following theorem which is a corollary \cite[Theorem 1.3]{Pis3} of a result of Beurling \cite[Theorem III]{Beu} (see also \cite[Theorem 2.1]{Pis1}, \cite[Corollary 2.5]{Fac} and \cite{Hin}).

\begin{thm}
\label{Théorème de Beurling} Let $X$ be a Banach space. Let
$(T_t)_{t\geq 0}$ be a strongly continuous semigroup of contractions
on $X$. Suppose that there exists some integer $n \geq 1$ such that
for any $t> 0$
\begin{equation*}\label{}
\bnorm{(Id_{X}-T_t)^n}_{X\to X} < 2^{n}.
\end{equation*}
Then the semigroup $(T_t)_{t\geq 0}$ is bounded analytic.
\end{thm}
Moreover, we recall the following lemma, see \cite[Lemma 13.12]{DJT} and \cite[Lemma 1.5]{Pis3}.
\begin{lemma}
\label{Lemma of Pisier on projections} 
Suppose that $X$ is a K-convex Banach space. Then there exist a real number $0<\rho<2$ and an integer $n\geq 1$ such
that if $P_1,\ldots, P_n$ is any finite collection of mutually
commuting norm one projections on $X$, then
\begin{equation*}
\Bgnorm{\prod_{1\leq k \leq n}(Id_{X}-P_k)}_{X\to X}\leq \rho^n.
\end{equation*}
\end{lemma}

We will use the useful next `absorption Lemma' which is a variant of \cite[Proposition 3.4]{Arh2}. The proof is left to the reader. 
\begin{lemma}
\label{Prop Fell absorption} Suppose $1\leq p < \infty$. Let $E$ be an operator space. For any positive integer $n\geq 1$ and any matrix $x \in \mathbb{M}_I(E)$ finitely supported on $I \times I$, we have 
\begin{equation}\label{Fell absorption}
\Bgnorm{\sum_{i,j \in I} e_{ij} \ot  e_{ij}\ot \cdots \ot
 e_{ij} \ot x_{ij}}_{S^p(B(\ell^2_I)^{\ot n},E)}=\Bgnorm{\sum_{i,j
\in I} e_{ij} \ot x_{ij}}_{S^p_I(E)}.
\end{equation}
Moreover, for any regular Schur multiplier
$M_A \colon S^p_I \to S^p_I$ and
any positive integer $n \geq 1$ we have
\begin{equation}\label{Majoration norme multiplicateur}
\bnorm{(M_A)^n \ot Id_E}_{S^p_I(E) \to S^p_I(E)}\leq
\bnorm{(M_A)^{\ot n} \ot Id_E}_{S^p(B(\ell^2_I)^{\ot n},E) \to S^p(B(\ell^2_I)^{\ot n},E)}.
\end{equation}
\end{lemma}

We also need the following transfer results of Neuwirth and Ricard \cite{NR} between Fourier multipliers and Schur multipliers. Let $G$ be a discrete group. If $\varphi \co G \to \C$ is a complex function, we denote by $\check{M}_{\varphi}$ the Schur multiplier defined by the matrix $\check{\varphi} \in \Mat_{G}$ defined by $\check{\varphi}(g,h)=\varphi(gh^{-1})$ where $g,h \in G$. This means that we have $\check{M}_{\varphi}=M_{\check{\varphi}}$. Moreover for any integer $n\geq 0$, we have 
\begin{equation}
\label{equavarphi}
  \big(\check{M}_{\varphi}\big)^n=\check{\overbrace{(M_{\varphi})^n}}.
\end{equation}
If $G$ is amenable, by \cite[page 1172]{NR}, we have for any $1 \leq p \leq \infty$
\begin{equation}
\label{Equa transfer leq}
\bnorm{M_\varphi \ot Id_E}_{L^p(\VN(G),E) \to L^p(\VN(G),E)}
\leq \bnorm{\check{M}_{\varphi}\ot Id_E}_{S^p_G(E) \to S^p_G(E)}
\end{equation}
and
\begin{equation}
\label{Equa transfer =}
\bnorm{M_\varphi \ot Id_E}_{cb,L^p(\VN(G),E) \to L^p(\VN(G),E)}
=\bnorm{ \check{M}_{\varphi}\ot Id_E}_{cb,S^p_G(E) \to S^p_G(E)}.
\end{equation}

%

Using the identification (\ref{Case abelian AVN}), we see that Theorem \ref{Th main commutative} is a particular case of the following more general result, which improves a part of \cite[Theorem 5.1]{Arh2}.

\begin{thm}
\label{Th main}
Suppose that $G$ is an amenable discrete group. Let $(T_t)_{t \geq 0}$ be a $w^*$-continuous semigroup of selfadjoint contractive Fourier multipliers on the group von Neumann algebra $\VN(G)$. Suppose that $E$ is an OK-convex operator space. Consider $1<p<\infty$. Then $(T_t)_{t\geq 0}$ induces a strongly continuous bounded analytic semigroup $(T_t \ot Id_E)_{t \geq 0}$ of contractions on the noncommutative vector valued $L^p$-space $L^p(\VN(G),E)$.
\end{thm}

\begin{proof}
We consider the associated semigroup $\big(\check{T_t}\big)_{t\geq 0}$ of selfadjoint contractive Schur multipliers on the space $B(\ell^2_G)$. Since $G$ is amenable, for any $t \geq 0$, the Fourier multiplier $T_t$ is completely contractive on $\VN(G)$. Using the part 2 of Proposition \ref{prop-tensorisation of CP maps}, we deduce that the map $T_t \ot Id_E$ extends to a complete contraction on the space $L^p(\VN(G),E)$. By (\ref{Equa transfer =}), we see that
\begin{align*}
\bnorm{\check{T_t} \ot Id_E}_{S^p_G(E) \to S^p_G(E)}
    &\leq \bnorm{\check{T_t}\ot Id_E}_{cb,S^p_G(E) \to S^p_G(E)}\\
		&=\bnorm{T_t \ot Id_E}_{cb,L^p(\VN(G),E) \to L^p(\VN(G),E)} \leq 1.
\end{align*}
In the sequel, we denote by $(\check{T_t})^\circ$ the Schur multiplier defined by the adjoint matrix of the matrix of the Schur multiplier $\check{T_t}$. As the proof of \cite[Corollary 4.3]{Arh1}, for any $t \geq 0$, there exists Schur multipliers $S_{1,t}$ and $S_{2,t}$ on $B(\ell^2_G)$ such that
$$
W_{t}
=
\left[
  \begin{array}{cc}
    S_{1,t} & \check{T_t} \\
    (\check{T_t})^\circ  & S_{2,t} \\
  \end{array}
\right]
$$
is a completely positive unital self-adjoint Schur multiplier on $B\big(\ell^2_{\{1,2\} \times G}\big)$. Using \ref{equavarphi}, for any $t\geq 0$, we see that
\begin{equation}
	\label{bloc}
	\big(W_{\frac{t}{2}}\big)^2
=\left[
  \begin{array}{cc}
    S_{1,\frac{t}{2}}      & \check{T}_{\frac{t}{2}} \\
    \big(\check{T}_{\frac{t}{2}}\big)^\circ  & S_{2,\frac{t}{2}} \\
  \end{array}
\right]^2
=\left[
  \begin{array}{cc}
    (S_{1,\frac{t}{2}})^2                        & (\check{T}_{\frac{t}{2}})^2 \\
    \big(\check{T}_{\frac{t}{2}}\big)^{\circ 2}  & (S_{2,\frac{t}{2}})^2 \\
  \end{array}
\right]
=
\left[
  \begin{array}{cc}
    (S_{1,\frac{t}{2}})^2  &   \check{T_{t}}         \\
    (\check{T}_{t})^\circ  &   (S_{2,\frac{t}{2}})^2 \\
  \end{array}
\right]
.
\end{equation}
Combining the construction of the noncommutative Markov chain of \cite[pages 4369-4370]{Ric} and the proof of \cite[Theorem 5.3]{HM}, for any $t \geq 0$, we infer that the Schur multiplier $(W_{\frac{t}{2}})^2$ admits a Rota dilation 
$$
\big((W_{\frac{t}{2}})^2\big)^{k}=Q\mathbb{E}_k\pi,\qquad k \geq 1
$$
in the sense of \cite[Definition 10.2]{JMX} (extended to semifinite von Neumann algebras) where $\pi \co M_2(B(\ell^2_{G})) \to M$ is a normal unital faithful $*$-representation into a von Neumann algebra (equipped with a trace) which preserve the traces, where $Q \co M \to M_2(B(\ell^2_{G}))$ is the conditional expectation associated with $\pi$ and where the $\mathbb{E}_k$'s are conditional expectations onto von Neumann subalgebras of $M$. Recall that the von Neumann algebra $\Gamma_{-1}^{e}(\ell^{2,T})$ of \cite{Ric} is hyperfinite. Hence, the von Neumann algebra $M$ of the Rota Dilation is also hyperfinite. Note that we only need the case $k=1$ in the sequel of the proof. In particular, we have
$$
(W_{\frac{t}{2}})^2=Q\mathbb{E}_1\pi.
$$
We infer that
$$
Id_{M_2(B(\ell^2_{G}))}-\big(W_{\frac{t}{2}}\big)^2=Q\pi-Q\mathbb{E}_1\pi=Q(Id_{M}-\mathbb{E}_1)\pi.
$$
Now, we choose an integer $n\geq 1$ and $0<\rho<2$ as in Lemma \ref{Lemma of Pisier on projections}. Note that we have
\begin{equation}
	\label{equa complexe}
	\Big(Id_{M_2(B(\ell^2_{G}))}-\big(W_{\frac{t}{2}}\big)^2\Big)^{\ot n}=Q^{\ot n}(Id_{M}-\mathbb{E}_1)^{\ot n}\pi^{\ot n}.
\end{equation}
For any integer $1 \leq k \leq n$, we consider the completely positive operator
$$
\Pi_k=Id_{L^p(M)} \ot \cdots \ot Id_{L^p(M)} \ot \mathbb{E}_1 \ot
Id_{L^p(M)} \ot \cdots \ot Id_{L^p(M)}
$$
on the space $L^p(M^{\ot n})$. By Proposition \ref{prop-tensorisation of CP maps}, we deduce that the $\Pi_k \ot Id_E$'s induce a family of mutually commuting contractive projections on the Banach space $L^p(M^{\ot n},E)$. Moreover, by \cite[Proposition 3.5]{Arh2}, the latter space is K-convex.  Hence, we obtain that
\begin{equation}
	\label{equa strange}
\Bgnorm{\prod_{1\leq k \leq n}  \Big(Id_{L^p(M^{\ot n},E)}-(\Pi_k\ot Id_E)\Big)}_{L^p(M^{\ot n},E) \to L^p(M^{\ot n},E)} \leq \rho^n.	
\end{equation}
Furthermore, we have
\begin{align}
\label{equazert}
\MoveEqLeft
\big(Id_{L^p(M)}-\mathbb{E}_1\big)^{\ot n}\\  
&= \prod_{1\leq k \leq n} \big(Id_{L^p(M)} \ot \cdots \ot Id_{L^p(M)} \ot (Id_{L^p(M)}-\mathbb{E}_1) \ot Id_{L^p(M)} \ot \cdots \ot Id_{L^p(M)}\big)\nonumber \\
		&=\prod_{1\leq k \leq n} \big(Id_{L^p(M^{\ot n})}-Id_{L^p(M)} \ot \cdots \ot Id_{L^p(M)} \ot \mathbb{E}_1 \ot Id_{L^p(M)} \ot \cdots \ot Id_{L^p(M)}\big)\nonumber\\
		&=\prod_{1\leq k \leq n} \big(Id_{L^p(M^{\ot n})}-\Pi_k\big).\nonumber
\end{align}
Now, combining (\ref{bloc}), (\ref{Majoration norme multiplicateur}), (\ref{equa complexe}), Proposition \ref{prop-tensorisation of CP maps}, (\ref{equazert}) and (\ref{equa strange}) we obtain that
\begin{align*}
\MoveEqLeft 
\Bnorm{\big(Id_{S^p_{G}}-\check{T}_t\big)^n \ot Id_E}_{S^p_{G}(E) \to S^p_{G}(E)}\leq \Bnorm{\Big(Id_{S^p_2(S^p_{G})}-\big(W_{\frac{t}{2}}\big)^2\Big)^n\ot Id_{E}}_{S^p_2(S^p_{G}(E)) \to S^p_2(S^p_{G}(E))}\\
&\leq \Bnorm{\Big(Id_{S^p_2(S^p_{\hat{G}})}-\big(W_{\frac{t}{2}}\big)^2\Big)^{\ot n}\ot Id_{E}}_{S^p(B(\ell^2_{\{1,2\} \times G})^{\ot n},E) \to S^p(B(\ell^2_{\{1,2\} \times G})^{\ot n},E)}\\
&= \Bnorm{Q^{\ot n}(Id_{L^p(M)}-\mathbb{E}_1)^{\ot n}\pi^{\ot n} \ot Id_E}_{S^p(B(\ell^2_{\{1,2\} \times G})^{\ot n},E) \to S^p(B(\ell^2_{\{1,2\} \times G})^{\ot n},E)}\\
    & \leq \Bnorm{\big(Id_{L^p(M)}-\mathbb{E}_1\big)^{\ot n}\ot Id_{E}}_{L^p(M^{\ot n},E) \to L^p(M^{\ot n},E)}\\
		& = \Bgnorm{\prod_{1\leq k \leq n}  \Big(Id_{L^p(M^{\ot n},E)}-(\Pi_k\ot Id_E)\Big)}_{L^p(M^{\ot n},E) \to L^p(M^{\ot n},E)} \leq \rho^n.
\end{align*}
Hence, using (\ref{Equa transfer leq}) and (\ref{equavarphi}), for any $t \geq 0$, we finally obtain
\begin{align*}
\MoveEqLeft
\Bnorm{\big(Id_{S^p_{G}}-T_t\big)^n \ot Id_E}_{L^p(\VN(G),E) \to L^p(\VN(G),E)}   
		\leq \bgnorm{\check{\overbrace{\big(Id_{S^p_{G}}-T_t\big)^n}} \ot Id_E}_{S^p_{G}(E) \to S^p_{G}(E)}\\
		&\leq \Bnorm{\big(Id_{S^p_{G}}-\check{T}_t\big)^n \ot Id_E}_{S^p_{G}(E) \to S^p_{G}(E)} \leq \rho^n.
\end{align*}
We conclude by Theorem \ref{Théorème de Beurling}.
\end{proof}


The next result is an improvement of \cite[Theorem 3.7]{Arh2} and can be proved as Theorem \ref{Th main}.

\begin{thm}
\label{Th main Schur}
Let $(T_t)_{t \geq 0}$ be a $w^*$-continuous semigroup of selfadjoint contractive Schur multipliers on $B(\ell^2_I)$. Suppose that $E$ is an OK-convex operator space. Consider $1<p<\infty$. Then $(T_t)_{t\geq 0}$ induces a strongly continuous bounded analytic semigroup $(T_t \ot Id_E)_{t \geq 0}$ of contractions on the vector valued Schatten space $S^p_I(E)$.
\end{thm}

\begin{remark}
Wo does not know if any K-convex Banach space $X$ is isomorphic to a Banach space $E$ admitting an operator space structure such that the Banach space $S^p(E)$ is K-convex for $1<p< \infty$ (i.e. $E$ is OK-convex). Moreover, it would be also interesting to examine a similar question for other Banach spaces properties (UMD, cotype...).  
\end{remark} 

\section{Semigroups of contractive selfadjoint Schur multipliers}

The description of self-adjoint contractive Schur multipliers on $B(\ell^2_I)$ (and more generally of contractive Schur multipliers) is well-known and essentially goes back to Grothendieck and was rediscovered by many authors, see \cite[Chapter 5]{Pis6} for more information. Here, we give a continuous version of this result which precisely describes the semigroups of Schur multipliers of Theorem \ref{Th main Schur} using ultraproducts, We refer to \cite[Chapter 8]{DJT} for more information on this notion. The proof illustrates the philosophy described in \cite{Tao}. A similar trick is used in the proof of \cite[Proposition 4.3]{Knu}.

\begin{thm}
\label{description Schur mult cp} 
Suppose that $A$ is a matrix of $\mathbb{M}_I$. For any $t\geq 0$, let $T_t$ be the unbounded Schur multipliers on $B\big(\ell^2_I\big)$ associated with the matrix
\begin{equation}
\label{matrix semigroup}
\Big[e^{-ta_{ij}}\Big]_{i,j \in I}.
\end{equation}
The semigroup $(T_t)_{t\geq 0}$ extends to a semigroup of selfadjoint contractive Schur multipliers \linebreak[4] $T_t \co B\big(\ell^2_I\big) \to B\big(\ell^2_I\big)$ if and only if there exists a Hilbert space $H$ and two families $(\alpha_i)_{i \in I}$ and $(\beta_j)_{j \in I}$ of elements of $H$ such that $a_{ij}=\norm{\alpha_i-\beta_j}_{H}^2$ for any $i,j \in I$.

In this case, the Hilbert space may be chosen as a real Hilbert
space and moreover, $(T_t)_{t\geq 0}$ is a w*-continuous semigroup.
\end{thm}

\begin{proof}
First, suppose that the semigroup $(T_t)_{t\geq 0}$ extends to a semigroup of selfadjoint contractive Schur multipliers $T_t \co B\big(\ell^2_I\big) \to B\big(\ell^2_I\big)$. In particular, for any integer $n\geq 1$, the Schur multiplier $T_{\frac{1}{n}} \co B\big(\ell^2_I\big) \to B\big(\ell^2_I\big)$ is contractive and selfadjoint. Thus, as the proof of \cite[Corollary 4.3]{Arh2}, we can find matrices $S_{1,n}$, $S_{2,n} \in  \mathbb{M}_{I}$ such that the block matrix
$$
B_n=\left[
  \begin{array}{cc}
    S_{1,n} & \big[e^{-\frac{a_{ij}}{n}}\big] \\
    \big[ e^{-\frac{a_{ji}}{n}}\big] & S_{2,n} \\
  \end{array}
\right] 
$$
defines a selfadjoint unital completely positive Schur multiplier on $B\big(\ell^2_{\{1,2\} \times I}\big)$. This matrix identifies to a matrix $[b_{n,k,i,m,j}]_{(k,i)\in\{1,2\}\times I ,(m,j) \in \{1,2\}\times I} \in \mathbb{M}_{\{1,2\}\times I}$ such that for any $i,j \in I$
\begin{align*}
b_{n,1,i,1,j}&= (S_{1,n})_{i,j}, &
b_{n,1,i,2,j} &= e^{-\frac{a_{ij}}{n}},\\
b_{n,2,i,1,j} &= e^{-\frac{a_{ji}}{n}}, &
b_{n,2,i,2,j} &= (S_{2,n})_{i,j}.
\end{align*}
Since the matrix $B_n$ defines a completely positive Schur multiplier, the map 
\begin{equation*}
\begin{array}{cccc}
&   (\{1,2\} \times I)\times (\{1,2\}\times I)  &  \longrightarrow   &  \R  \\
 &  \displaystyle \big((k,i),(m,j)\big)   &  \longmapsto       & b_{n,k,i,m,j}   \\
\end{array}
\end{equation*} is a positive definite kernel. Then it is obvious that the map 
\begin{equation*}
\begin{array}{cccc}
&   (\{1,2\} \times I)\times (\{1,2\}\times I)  &  \longrightarrow   &  \R  \\
 &  \displaystyle \big((k,i),(m,j)\big)   &  \longmapsto       & n(1-b_{n,k,i,m,j})   \\
\end{array}
\end{equation*}
is a real-valued conditionally negative definite kernel which vanishes on the diagonal of $(\{1,2\} \times I)\times (\{1,2\}\times I)$. We deduce that there exist a real Hilbert space $H_n$ and a map $\xi^n \co \{1,2\}\times I \to H_n$ such that
$$
\bnorm{\xi_{k,i}^n-\xi_{m,j}^n}_{H_n}^2 = n\big(1-b_{n,k,i,m,j}\big),\quad k,m\in \{1,2\},\ i,j \in I. 
$$
By adding a constant to $\xi^n$, it is not difficult to see that we can suppose that $\xi_{2,1}^n=0$ for any integer $n \geq 1$. For any $a \in \R$, we have $\frac{1- e^{-ta}}{t} \xra[t \to 0]{} a$. Hence, for any $i,j \in I$, we see that 
$$
\bnorm{\xi_{1,i}^n-\xi_{2,j}^n}_{H_n}^2 = n\big(1-b_{n,1,i,2,j}\big) = n\big(1- e^{-\frac{a_{ij}}{n}}\big) \xra[n \to \infty]{} a_{ij}.
$$
Since $\xi_{2,1}^n=0$, we obtain in particular $\bnorm{\xi_{1,i}^n}_{H_n}^2 \xra[n \to \infty]{} a_{i1}$ for any integer $i \in I$. We infer that $\big(\norm{\xi_{1,i}^n}_{H_n})_{n\geq 1}$ is a bounded sequence for each $i \in I$. Moreover, note that for any $i,j \in I$ we have
$$
\bnorm{\xi_{2,j}^n}_{H_n} \leq \bnorm{\xi_{2,j}^n-\xi_{1,i}^n}_{H_n}+\bnorm{\xi_{1,i}^n}_{H_n}.
$$
Thus, for any $j \in I$, we deduce that $\big(\norm{\xi_{2,j}^n}_{H_n})_{n\geq 1}$ is also a bounded sequence.

Now, we introduce the ultraproduct $H=\prod_{\ul}H_n$ of the Hilbert spaces $H_n$ with respect to some ultrafilter $\ul$ on $\N$ refining the Fr\'echet filter. For any $k \in \{1,2\}$ and any $i \in I$, let $\xi_{k,i} \in H$ the equivalence class of the sequence $\big(\xi_{k,i}^n\big)_{n \geq 1}$. The above computations give
\begin{align}
\label{equa 33}
\bnorm{\xi_{1,i}-\xi_{2,j}}_{H}^2=a_{ij}, \quad  i,j \in I.
\end{align}
For any $i,j \in I$ we let $\alpha_i = \xi_{1,i}$ and $\beta_j = \xi_{2,j}$. Then Equation~\eqref{equa 33} becomes
$$
\bnorm{\alpha_i-\beta_j}_{H}^2=a_{ij}, \quad i,j \in I.
$$
Conversely, suppose that there exists a Hilbert space $H$ and two
families $(\alpha_i)_{i \in I}$ and $(\beta_j)_{j \in I}$ of
elements of $H$ such that for any $t\geq 0$ the Schur multiplier
$T_t \co B\big(\ell^2_I\big) \to B\big(\ell^2_I\big)$ is associated
with the matrix
$$
A_t=\Big[e^{-t\norm{\alpha_i-\beta_j}_{H}^2}\Big]_{i,j \in I}.
$$
Now, for any $t\geq 0$, we define the following matrices of
$\mathbb{M}_I$
$$
B_t=\Big[e^{-t\norm{\alpha_i-\alpha_j}_{H}^2}\Big]_{i,j\in I},\ \
C_t=\Big[e^{-t\norm{\beta_i-\beta_j}_{H}^2}\Big]_{i,j \in I}\ \
\text{ and }\ \
D_t=\Big[e^{-t\norm{\beta_i-\alpha_j}_{H}^2}\Big]_{i,j \in I}.
$$
For any $i\in I$ and any $n\in \{1,2\}$, we define the vector
$\gamma_{(n,i)}$ of $H$ by
$$
\gamma_{(n,i)}=\left\{
\begin{array}{cl}
         \alpha_i & \text{if}\  n=1   \ \text{and}\  i \in I       \\
         \beta_i  & \text{if}\  n=2   \ \text{and}\  i \in I.        \\
\end{array}\right.
$$
Now, by the identification $\mathbb{M}_2(\mathbb{M}_I)\simeq
\mathbb{M}_{\{1,2\}\times I}$, the block matrix $ \left[
  \begin{array}{cc}
    B_t & A_t \\
    D_t & C_t \\
  \end{array}
\right] $ of $\mathbb{M}_2(\mathbb{M}_I)$ can be identified with the matrix
\begin{equation}
	\label{matrice}
	\Big[ e^{-t\norm{\gamma_{(n,i)}-\gamma_{(m,j)}}_{H}^2}
\Big]_{(n,i)\in\{1,2\}\times I ,(m,j)\in\{1,2\}\times I}
\end{equation}
of $\mathbb{M}_{\{1,2\}\times I}$. Using \cite[Proposition 5.4]{Arh1}, we deduce that, for any $t\geq 0$, the Schur multiplier associated with the matrix (\ref{matrice}) is contractive on $B\big(\ell^2_{\{1,2\}\times I}\big)$. We deduce that
$T_t \co B\big(\ell^2_I\big) \to B\big(\ell^2_I\big)$ is also
contractive. An alternative proof of this implication can be achieved by adapting the proof of \cite[Proposition 8.17]{JMX}.

Finally, it is easy to see that $(T_t)_{t\geq 0}$ is a w*-continuous semigroup.
\end{proof}

\textbf{Acknowledgment}. The author would like to thank to Christian Le Merdy to provide him with the preprint \cite{Xu} and Marius Junge and Eric Ricard for some discussions. Finally, the referee deserves thanks for a careful reading of the paper.


\footnotesize{ \n Laboratoire de Math\'ematiques, Universit\'e de
Franche-Comt\'e,
25030 Besan\c{c}on Cedex,  France\\
cedric.arhancet@univ-fcomte.fr\hskip.3cm

\small
\begin{thebibliography}{79}

\bibitem[Arh1]{Arh1}
C. Arhancet.
\newblock On Matsaev's conjecture for contractions on noncommutative $L^p$-spaces.
\newblock Journal of Operator Theory 69 (2013), no. 2, 387--421.

\bibitem[Arh2]{Arh2}
C. Arhancet.
\newblock Analytic semigroups on vector valued noncommutative $L^p$-spaces.
\newblock Studia Math. 216 (2013), no. 3, 271--290.

\bibitem[Arh3]{Arh3}
C. Arhancet.
\newblock Semigroups of operators and OK-convexity.
\newblock In preparation.

\bibitem[ABHN]{ABHN}
W. Arendt, C. J. K. Batty, M. Hieber, F. Neubrander.
\newblock Vector-valued Laplace transforms and Cauchy problems. Second edition.
\newblock Monographs in Mathematics, 96. Birkhäuser/Springer Basel AG, Basel, 2011.

\bibitem[BCR]{BCR}
C. Berg, J. Christensen and P. Ressel.
\newblock Harmonic analysis on semigroups. Theory of positive definite and related functions.
\newblock Springer-Verlag, New York, 1984.


\bibitem[BHV]{BHV}
B. Bekka, P. de la Harpe and A. Valette.
\newblock Kazhdan's property (T).
\newblock New Mathematical Monographs, 11. Cambridge University Press, Cambridge, 2008.

\bibitem[Beu]{Beu}
A. Beurling.
\newblock On analytic extension of semigroups of operators.
\newblock J. Funct. Anal. 6 (1970), 387--400 .









\bibitem[DCH]{DCH}
J. De Cannière and U. Haagerup.
\newblock Multipliers of the Fourier algebras of some simple Lie groups and their discrete subgroups.
\newblock Amer. J. Math. 107 (1985), no. 2, 455--500.


\bibitem[DJT]{DJT}
J. Diestel, H. Jarchow, and A. Tonge.
\newblock Absolutely summing operators.
\newblock Cambridge Studies in Advanced Mathematics 43, Cambridge University Press, Cambridge, 1995. 

\bibitem[EN]{EN}
K.-J. Engel and R. Nagel.
\newblock One-parameter semigroups for linear evolution equations.
\newblock Graduate Texts in Mathematics, 194. Springer-Verlag, New York, 2000. 

\bibitem[ER]{ER}
E. Effros and Z.-J. Ruan.
\newblock Operator spaces.
\newblock Oxford University Press (2000).

\bibitem[Fac]{Fac}
S. Fackler.
\newblock Regularity of semigroups via the asymptotic behaviour at zero.
\newblock Semigroup Forum 87 (2013), no. 1, 117.

\bibitem[HM]{HM}
U. Haagerup and M. Musat.
\newblock Factorization and dilation problems for completely positive maps on von Neumann algebras.
\newblock Comm. Math. Phys. 303 (2011), no. 2, 555--594.


\bibitem[Haa]{Haa}
M. Haase.
\newblock The functional calculus for sectorial operators.
\newblock Operator Theory: Advances and Applications, 169. Birkhäuser Verlag (2006).


\bibitem[Hin]{Hin}
A. Hinrichs.
\newblock $K$-convex operators and Walsh type norms.
\newblock  Math. Nachr. 208 (1999), 121--140.




\bibitem[Knu]{Knu}
S. Knudby.
\newblock Semigroups of Herz-Schur multipliers.
\newblock J. Funct. Anal. 266 (2014), no. 3, 1565--1610.

\bibitem[JMX]{JMX}
M. Junge, C. Le Merdy and Q. Xu.
\newblock $H^\infty$ functional calculus and square functions on noncommutative $L^p$-spaces.
\newblock Astérisque No. 305 (2006).

\bibitem[JP]{JP}
M. Junge and J. Parcet.
\newblock The norm of sums of independent noncommutative random variables in $L_p(l_1)$.
\newblock J. Funct. Anal. 221 (2005), no. 2, 366--406.

\bibitem[JX]{JX}
M. Junge and Q. Xu.
\newblock Noncommutative maximal ergodic theorems.
\newblock J. Amer. Math. Soc. 20 (2007), no. 2, 385--439.




\bibitem[Los]{Los}
V. Losert.
\newblock Properties of the Fourier algebra that are equivalent to amenability.
\newblock Proc. Amer. Math. Soc. 92 (1984), no. 3, 347--354.

\bibitem[NR]{NR}
S. Neuwirth and \'E. Ricard.
\newblock Transfer of Fourier multipliers into Schur multipliers and sumsets in a discrete group.
\newblock Canad. J. Math. 63 (2011), no. 5, 1161--1187.  



\bibitem[Pau]{Pau}
V. Paulsen.
\newblock Completely bounded maps and operator algebras.
\newblock Cambridge Univ. Press (2002).


\bibitem[Pis1]{Pis1}
G. Pisier.
\newblock Semi-groupes holomorphes et $K$-convexité. (French) [Holomorphic semigroups and $K$-convexity].
\newblock Seminar on Functional Analysis, 1980--1981, Exp. No. II, 32 pp., École Polytech., Palaiseau, 1981. 

\bibitem[Pis2]{Pis2}
G. Pisier.
\newblock Semi-groupes holomorphes et $K$-convexité (suite). (French) [Holomorphic semigroups and $K$-convexity (continued)]
\newblock Seminar on Functional Analysis, 1980--1981, Exp. No. VII, 10 pp., École Polytech., Palaiseau, 1981. 

\bibitem[Pis3]{Pis3}
G. Pisier.
\newblock Holomorphic semigroups and the geometry of Banach spaces.
\newblock Ann. of Math. (2) 115 (1982), no. 2, 375--392.

\bibitem[Pis4]{Pis4}
G. Pisier.
\newblock Regular operators between non-commutative $L_p$-spaces.
\newblock Bull. Sci. Math. 119 (1995), no. 2, 95--118.

\bibitem[Pis5]{Pis5}
G. Pisier.
\newblock Non-commutative vector valued $L_p$-spaces and completely $p$-summing maps.
\newblock Astérisque, 247 (1998).

\bibitem[Pis6]{Pis6}
G. Pisier.
\newblock Similarity problems and completely bounded maps, volume 1618 of Lecture Notes in Mathematics.
\newblock Springer-Verlag, expanded edition, 2001.

\bibitem[Pis7]{Pis7}
G. Pisier.
\newblock Introduction to operator space theory.
\newblock Cambridge University Press, Cambridge (2003).

\bibitem[Pis8]{Pis8}
G. Pisier.
\newblock Remarks on the non-commutative Khintchine inequalities for $0<p<2$.
\newblock J. Funct. Anal. 256 (2009), no. 12, 4128--4161.

\bibitem[PX]{PiX}
G. Pisier and Q. Xu.
\newblock Non-commutative $L^p$-spaces.
\newblock 1459--1517 in Handbook of the Geometry of Banach Spaces, Vol. II, edited by W.B. Johnson and J. Lindenstrauss, Elsevier (2003).

\bibitem[Ric]{Ric}
\'E. Ricard.
\newblock A Markov dilation for self-adjoint Schur multipliers.
\newblock Proc. Amer. Math. Soc. 136 (2008), no. 12, 4365--4372.

\bibitem[SS]{SS}
A. Sinclair and R. Smith.
\newblock Finite von Neumann algebras and masas.
\newblock London Mathematical Society Lecture Note Series, 351. Cambridge University Press, Cambridge, 2008.

\bibitem[Spr]{Spr}
N. Spronk.
\newblock Measurable Schur multipliers and completely bounded multipliers of the Fourier algebras.
\newblock Proc. London Math. Soc. (3) 89 (2004), no. 1, 161--192.

\bibitem[Ste]{Ste}
E. M. Stein.
\newblock Topics in harmonic analysis related to the Littlewood-Paley theory.
\newblock Annals of Mathematics Studies, No. 63 Princeton University Press,
Princeton, N.J.; University of Tokyo Press, Tokyo (1970).

\bibitem[Str]{Str}
S. Stratila.
\newblock Modular theory in operator algebras.
\newblock Taylor and Francis, 1981.

\bibitem[Tao]{Tao}
T. Tao.
\newblock Ultraproducts as a Bridge Between Discrete and Continuous Analysis.
\newblock \href{https://terrytao.wordpress.com/2013/12/07/ultraproducts-as-a-bridge-between-discrete-and-continuous-analysis/}{https://terrytao.wordpress.com/2013/12/07/ultraproducts-as-a-bridge-between-discrete-and-continuous-analysis/}.

\bibitem[Xu]{Xu}
Q. Xu.
\newblock $H^\infty$ functional calculus and maximal inequalities for semigroups of contractions on vector-valued $L_p$-spaces.
\newblock Preprint, arXiv:1402.2344.

\end{thebibliography}
\end{document}